\documentclass[11pt]{amsart}
\usepackage{amsmath,amssymb,amsthm}
\usepackage[latin1]{inputenc}

\usepackage[sectionbib,numbers]{natbib}
\usepackage{mathrsfs}
\usepackage[english]{babel}
\usepackage{graphicx}

\numberwithin{equation}{section}  

\begin{document}

\renewcommand{\theenumi}{(\roman{enumi})}

\newcommand{\mydp}[1]{\displaystyle{#1}}
\newcommand{\bs}[1]{\boldsymbol{#1}}
\def\widebar#1{\overline{#1}}

\newtheorem{lemma}{Lemma}
\newtheorem{cor}{Corollary}
\newtheorem{thm}{Theorem}
\newtheorem{prop}{Proposition}
\renewcommand{\theequation}{\thesection.\arabic{equation}}
\renewcommand{\thelemma}{\thesection.\arabic{lemma}}
\renewcommand{\theprop}{\thesection.\arabic{prop}}
\renewcommand{\thethm}{\thesection.\arabic{thm}}
\renewcommand{\thecor}{\thesection.\arabic{cor}}
\newcommand\newsec{%
\setcounter{equation}{0}%
\setcounter{lemma}{0}%
\setcounter{prop}{0}%
\setcounter{thm}{0}%
\setcounter{cor}{0}%
}%

\renewcommand{\theequation}{\thesection.\arabic{equation}}

\newcommand{\al}{\alpha}
\newcommand{\om}{\omega}
\newcommand{\Om}{\Omega}
\newcommand{\La}{\Lambda}
\newcommand{\la}{\lambda}
\newcommand{\De}{\Delta}
\newcommand{\de}{\delta}
\newcommand{\ep}{\epsilon}
\newcommand{\be}{\beta}
\newcommand{\ga}{\gamma}
\newcommand{\Ga}{\Gamma}
\newcommand{\te}{\theta}
\newcommand{\Te}{\Theta}
\newcommand{\si}{\sigma}
\newcommand{\Si}{\Sigma}
\newcommand{\vep}{\varepsilon}
\newcommand{\veps}{\varepsilon}
\newcommand{\eps}{\epsilon}
\newcommand{\ze}{\zeta}
\newcommand{\T}{\mathrm{\tiny T}} 
\newcommand{\pa}{\partial}
\newcommand{\dd}[2]{\frac{\pa #1}{\pa #2}} 

\newcommand{\Z}{\mathbb{Z}}
\newcommand{\R}{\mathbb{R}}
\newcommand{\CC}{\mathbb{C}}   
\newcommand{\Q}{\mathbb{Q}}
\newcommand{\N}{\mathbb{N}}

\def\cov{\mathop{\hbox{\rm cov}}}
\def\var{\mathop{\hbox{\rm var}}}

\newcommand{\PP}{\; \mathbb{P}}
\newcommand{\E}{\;  \mathbb{E}}

\newcommand{\rw}{\rightarrow}
\newcommand{\lrw}{\longrightarrow}
\newcommand{\Rw}{\Rightarrow}
\newcommand{\Lrw}{\Longrightarrow}


\def\minmax{\mathop{\hbox{\rm min\,max}}}
\def\tr{\mathop{\hbox{\rm tr}}}
\def\ln{\mathop{\hbox{\rm $\ell n$}}}
\def\Arg{\mathop{\hbox{\rm Arg}}}
\def\diag{\mathop{\hbox{\rm diag}}}
\def\grad{\mathop{\nabla}}
\def\card{\mathop{\hbox{\rm card}}}
\def\bc{\mathop{\hbox{\mbox{\large\rm $|$}}}}
\def\telque{\; : \;} 
\def\pds#1#2{{\langle #1,#2 \rangle}}         
\def\bpv{\mathop{;}}                          
\def\dfrac#1#2{\frac{\displaystyle #1}{\displaystyle #2}}
\def\ffrac#1{\frac{1}{#1}}
\newcommand{\norm}[2]{ \| #1 \|_{#2} }
\def\un{\mbox{\large 1\hskip-0.30em I}} 
\def\indic#1{\un_{#1}}                        
\def\cf{{\em cf. }} 
\def\sgn{\mathop{\hbox{\rm sgn}}}
\newcommand{\rmx}[1]{{\mbox{\rm #1}}}
\newcommand{\rmxs}[1]{{\mbox{\rm {\scriptsize #1}}}}
\newcommand{\rmxt}[1]{{\mbox{\rm {\tiny #1}}}}
\newcommand{\trsp}[1]{{{#1}^\T}}                 
\newcommand{\fq}[2]{\trsp{#1} #2 #1}

\newcommand{\eq}[1]{(\ref{#1})}         
\newcommand{\cte}{\rmx{cte }}                

\def\slabel#1{\label{#1} \hfill \hskip2mm\hbox to 4mm{\dotfill}%
        \mbox{\tiny\em(#1)}\hbox to 4mm{\dotfill}}

\def\ddemo#1{\noindent{\textbf{Proof {#1}. \quad}}}
\def\findemo{\hskip3mm\mbox{\quad\vbox{\hrule height 3pt depth 6pt width 6pt}}}

\renewcommand{\theenumi}{\roman{enumi}}
\newcommand\cvp{\stackrel{P}{\rw}}
\newcommand\cvps{\stackrel{a.s.}{\rw}}
\def\cvloi{\stackrel{{\mathcal{D}}}{\Longrightarrow}}
\newcommand\Cb{\mathscr{C}_b}
\newcommand\bbx{{\bf x}}
\newcommand\bbu{{\bf u}}
\newcommand\CN{\mathscr{N}}
\newcommand\cK{\mathscr{K}}
\newcommand{\bars}{{\underline s}}
\newcommand\barSn{{\underline S}_n}
\newtheorem{defi}{Definition}[section]
\newcommand{\bbeta}{{\mbox{\boldmath$\beta$}}}
\newcommand{\bbzeta}{{\mbox{\boldmath$\zeta$}}}
\newcommand{\bbtheta}{{\mbox{\boldmath$\theta$}}}
\newcommand{\bdeta}{{\mbox{\boldmath$\eta$}}}
\newcommand{\bgl}{{\mbox{\boldmath$\lambda$}}}
\newcommand{\bbell}{{\mbox{\boldmath$\ell$}}}
\newcommand{\bbxi}{{\mbox{\boldmath$\xi$}}}

\title[Generalized spiked population model]{Limit theorems for sample eigenvalues in a generalized spiked population model}
\author{Zhidong \textsc{Bai}}
\address{Zhidong Bai \\ 
KLASMOE and School of Mathematics and Statistics\\
Northeast Normal University \\
  5268 People's Road\\
  130024 Changchun, China\\
    and\\
    Department of Statistics and Applied Probability\\
    National University of Singapore\\
    10, Kent Ridge Crescent\\
	Singapore 119260
}
\email{stabaizd@nus.edu.sg}
\thanks{The research of this Zhidong \textsc{Bai} was supported by CNSF grant 10571020 and NUS grant R-155-000-061-112}
    
\author{{Jian-feng} \textsc{Yao}}
\address{Jian-feng Yao \\
  IRMAR and Universit\'e de Rennes 1\\
  Campus de Beaulieu\\
  35042   Rennes Cedex, France\\
}
\email{jian-feng.yao@univ-rennes1.fr}
\thanks{Research was (partially) completed while J.-F. \textsc{Yao}
    was visiting 
    the Department of Statistics and Applied Probability,  National University of Singapore
    in March 2007}


\begin{abstract}
     \quad   In the  spiked population model introduced by
     \citet{Johnstone01}, 
     the population covariance matrix has all its eigenvalues 
     equal to unit except for a few fixed eigenvalues (spikes).  
     The question is to quantify the effect of the perturbation caused by the spike
	 eigenvalues.  \citet{BaikSilv06}  establishes 
	 the almost sure limits of the extreme sample eigenvalues associated to the
	 spike eigenvalues when the population and the sample sizes become
	 large.  In a recent work \cite{BaiYaoIHP}, we have provided  the limiting 
	 distributions for  these extreme sample eigenvalues. 
     In this paper, we extend this theory to a {\em generalized} spiked
     population model where the base population covariance matrix is
     arbitrary, instead of the identity matrix as in Johnstone's
     case. New mathematical tools are introduced for establishing the almost sure
     convergence of the sample eigenvalues generated by the spikes. 
\end{abstract}

\maketitle 

\section{Introduction} \label{sec:intro}
Let $(T_p)$ be a sequence of $p\times p$ non-random and nonnegative
definite Hermitian matrices and let $(w_{ij})$, $i,j\ge 1$ be a
doubly infinite array of i.i.d. complex-valued random variables
satisfying
\[  \label{eq:moments}
  \E(w_{11})=0,~~
  \E(|w_{11}|^2)=1,~~
  \E(|w_{11}|^4)<\infty.
\]
Write $Z_n=(w_{ij})_{1\le i\le p,1\le j \le n}$, the upper-left
$p\times n$ bloc,  where $p=p(n)$ is
related to $n$ such that when $n\rw \infty$, $p/n\rw y >0$.
 Then the matrix
$S_n=\frac1n T_p^{1/2} Z_nZ_n^*T_p^{1/2}$ can be considered as the
sample covariance matrix of an i.i.d. sample
$(\bbx_1,\ldots,\bbx_n)$ of $p$-dimensional observation vectors
$\bbx_j=T_p^{1/2}{\bbu_j}$ where ${\bbu_j}=(w_{ij})_{1\le i\le p}$
denotes the $j$-th column of $Z_n$. Throughout the paper, $A^{1/2}$
stands for any Hermitian square root of an nonnegative definite
(n.n.d.) Hermitian matrix $A$.

Assume that the empirical spectral
distribution (ESD) of $T_p$ converges weakly to a nonrandom
probability distribution $H$ on $[0,\infty)$. It is then well-known
that the ESD of $S_n$ converges to a nonrandom limiting spectral
distribution (LSD) $G$ \citep{MP,Silverstein95}.

Let $\la_{n,1} \ge \cdots\ge\la_{n,p}$ be the set of sample
eigenvalues, i.e. the eigenvalues of the sample covariance matrix
$S_n$.
The so-called {\em null case} corresponds to the situation $T_p\equiv
I_p$, so that, assuming $y\le 1$, the LSD $G$  reduces to the Mar\v{c}enko-Pastur law with
support $\Ga_G=[a_y,b_y]$ where $ a_y=(1-\sqrt y)^2$ and
$b_y=(1+\sqrt y)^2$.
Furthermore, the extreme sample eigenvalues
$\la_{n,1}$  and $\la_{n,p}$ almost surely tend to  $b_y$
and  $a_y$, respectively,  and
 the sample
eigenvalues $(\la_{n,j})$ fill completely the interval $[a_y,b_y]$.
However, as pointed out by  \citet{Johnstone01}, many empirical data
sets demonstrate a significant deviation from this null case since
some of sample  extreme eigenvalues are well separated from an
inner bulk interval. As a way for possible explanation of such
phenomenon, Johnstone proposes a {\em spiked population model}
where all eigenvalues of $T_p$ are unit except a fixed and relatively
small number among them ({\em spikes}).
In other words,  the population eigenvalues $\{\be_{n,j}\}$ of $T_p$ are
\[
  \underbrace{\al_1,\ldots,\al_1}_{n_1},
  \ldots,
  \underbrace{\al_K,\ldots,\al_K}_{n_K},
  \underbrace{ 1,\ldots,1}_{p-M},
\]
where $M$ is fixed as well as the multiplicity numbers $(n_k)$ which
satisfy $n_1+\cdots+n_K=M$. Clearly, this  spiked population model
can be viewed as a {finite-rank  perturbation of the
null case}.

Obviously, the LSD $G$ of  $S_n$ is not affected by this small
perturbation,  still equals  to the Mar\v{c}enko-Pastur law.
However,  the asymptotic behavior of the extreme eigenvalues of
$S_n$ is significantly different from  the  { null case}. The
fluctuation of  the largest eigenvalue $\la_{n,1}$ in case of
complex Gaussian variables has been recently studied in
\citet{BBP05}. These authors prove a transition phenomenon:   the
weak limit as well as the scaling of $\la_{n,1}$ is different
according to   its location with respect to a critical value
$1+\sqrt y$. In \citet{BaikSilv06}, the authors consider the spiked
population model with general random variables:  complex or real and
not necessarily Gaussian. For   the almost sure limits of the
extreme sample eigenvalues, they   also find   that these limits
depend on the critical values $1+\sqrt y$ for largest sample
eigenvalues, and on $1-\sqrt y$ for smallest ones. For example,  if
there are $m$  eigenvalues in the population covariance matrix
larger than $1+\sqrt y$, then the $m$ largest sample eigenvalues
$\la_{n,1},\ldots,\la_{n,m}$
 will converge to a limit above  the right edge $b_y$  of the limiting
Mar\v{c}enko-Pastur law, see \S\ref{ssec:Johnstone} for more details.
 In a recent work \citet{BaiYaoIHP},
considering general random matrices  as in
\cite{BaikSilv06}, we have established central limit theorems for
these extreme sample eigenvalues generated by spike eigenvalues
which are outside the critical interval $[1-\sqrt y,1+\sqrt y]$.

The spiked population model has also an extension to other random
matrices ensembles through the general concept of small-rank
perturbations. The goal is again to examine the effect caused on
the sample extreme   eigenvalues by such perturbations. 
In a series of recent papers \cite{Peche06,FerPec07,CapDonFer07},
these authors establish several results in this vein 
for  ensembles of form  
$M_n = W_n + n^{-1/2}V$ where $W_n$ is a standard Wigner matrix and
$V$ a small-rank matrix.

The present work is motivated by a generalization of Johnstone's
spike population model defined as  follows. The population
covariance matrix $T_p$ posses two sets of eigenvalues:  a small
number of them, say $(\al_k)$,  called {\em
  generalized spikes},
are well separated - in a sense to be defined later-,  from  a base
set $(\be_{n,i})$.
   In other words,
the spectrum of $T_p$ reads as
\[
  \underbrace{\al_1,\ldots,\al_1}_{n_1},
  \ldots,
  \underbrace{\al_K,\ldots,\al_K}_{n_K},
   \be_{n,1},\ldots,\be_{n,p-M}.
\]
Therefore, this scheme can be viewed as a {finite-rank
perturbation of a general population  covariance matrix} with
eigenvalues $\{\be_{n,j}\}$.

The empirical distributions generated by
the eigenvalues $(\be_{n,i})$ will be assumed to have a
 limit distribution  $H$.
Note that $H$ is also the LSD of $T_p$ since the perturbation is of finite rank.
Analogous to Johnstone's spiked population model,
the LSD $G$ of the sample covariance matrix $S_n$ is still not affected by the
spikes.
The aim of this work is to identify the effect  caused by the spikes $(\al_k)$ on a particular
subset of sample eigenvalues.  The results obtained here  extend
those  of \citep{BaikSilv06,BaiYaoIHP} to the present generalized
scheme.

The remaining sections of the paper are   organized as following.
\S \ref{sec:model} gives the precise definition of the
generalized spiked population model.
Next, we use \S\ref{sec:known} to recall several useful results on the
convergence of the E.S.D. from general sample covariance matrices.
 In \S\ref{sec:asconv}, we examine the strong point-wise convergence
of sample eigenvalues associated to spikes.
We then establish CLT for these sample eigenvalues in \S \ref{sec:CLT}
using the methodology developed in \cite{BaiYaoIHP}.
Preliminary lemmas and their proofs are gathered  in the last
section.

\section{Generalized spiked population model}
\label{sec:model}
In a generalized spiked population model,   the population
covariance matrix $T_p$  takes the form
\[ T_p=
\begin{pmatrix}
  \Si  &  0 \\
   0   &  V_{p}
\end{pmatrix},
\]
where $\Si$ and $V_{p}$ are nonnegative and nonrandom Hermitian
matrices of dimension $M\times M$ and $p'\times p'$,
respectively, where $p'=p-M$.
 The submatrix $\Si$ has $K$ eigenvalues
$\al_1>\cdots>\al_K>0$ of respective  multiplicity  $(n_k)$,
and $V_p$ has $p'$ eigenvalues
$\be_{n,1}\ge \cdots \ge \be_{n,p'}$.

Throughout the paper, we assume that the following  assumptions hold.
\begin{enumerate}
\item[(a)]
  $w_{ij}$, $i,j=1,2,...$ are i.i.d. complex random variables
  with $E w_{11}=0$, $E|w_{11}|^2=1$, and
  $E|w_{11}|^4<\infty$.
\item[(b)]
   $n=n(p)$ with    $y_n=p'/n\to y>0$ as $n\to\infty$.
\item[(c)]
  The sequence of  ESD  $H_n$ of $(T_p)$, i.e. generated by the
  population eigenvalues $\{\al_k,\be_{n,j} \} $,
  weakly converges to a probability distribution  $H$
  as $n\to\infty$.
\item[(d)]
  The sequence $(\|T_p\|)$ of  spectral norms of $(T_p)$ is bounded.
\end{enumerate}

For any measure $\mu$ on $\R$, we denote by  $\Gamma_\mu$ the   support of $\mu$, a
close set.

\begin{defi}
  An  eigenvalue $\al$ of the matrix $\Si$ is called a {\em
  generalized spike eigenvalue} if  $\al\notin \Ga_H$.
\end{defi}

To avoid confusion between spikes and non-spike eigenvalues, we
further assume that
\begin{enumerate}
\item[(e)]
$ \max\limits_{1\le j\le p'}d(\beta_{nj},\Gamma_H)=\varepsilon_n\to0$,
\end{enumerate}
where $d(x,A)$ denotes the distance of a  point $x$ to a set
$A$.
Note that there is a positive constant $\delta$ such that
$d(\al_k,\Gamma_H)>\delta$, for all $k\le K$.

The above definition for generalized spikes is
 consistent with Johnstone's original one of
(ordinary) spikes, since in that case we have $H_n\equiv
H=\de_{\{1\}}$ and $\al\notin \Ga_H$ simply means $\al\ne 1$.

Let us decompose the observation vectors $\bbx_j=T_p^{1/2}{\bbu_j}$, $j=1,\ldots,n$,
where ${\bbu_j}=(w_{ij})_{1\le i\le p}$ 
by blocs,
\[  \bbx_j=
\begin{pmatrix}
   \bbxi_j  \\ \bdeta_j
\end{pmatrix},
\quad \textrm{with} \quad \bbxi_j= \Si^{1/2} (w_{ij})_{1\le i\le
M},\quad \bdeta_j= V_{p}^{1/2} (w_{ij})_{M<i\le p}.
\]
Note that both sequences $\{\bbxi_1,\ldots, \bbxi_n\}$ and
$\{\bdeta_1,\ldots, \bdeta_n\}$ are i.i.d. sequences. We also denote
the coordinates of $\bbxi_1$ by $\bbxi_1=(\xi(1),\ldots,\xi(M))^T$.

Similarly, the  sample covariance matrix $S_n=\frac1n
T_p^{1/2}Z_nZ_n^*T_p^{1/2}$ is decomposed as
\[
  S_n=\begin{pmatrix}
  S_{11}&S_{12}\\
  S_{21}&S_{22}
  \end{pmatrix}
  =
  \begin{pmatrix}
    X_1X_1^*&X_1X_2^*\\  X_2X_1^* &X_2X_2^*
  \end{pmatrix}~,
\]
with
\[
X_1=\frac1{\sqrt n} (\bbxi_1,\cdots,\bbxi_n)_{M\times n}
=\frac1{\sqrt n}\bbxi_{1:n} , \ X_2=\frac1{\sqrt n}
(\bdeta_1,\cdots,\bdeta_n)_{p'\times n}= \frac1{\sqrt n}\bdeta_{1:n}
~.
\]

Throughout the paper and for any Hermitian matrix $A$, we order its
eigenvalues in an descending order as $\la_1^A \ge\la_2^A\ge
\cdots.$ By definition, the sample eigenvalues $
\{\la^{S_n}_{j},1\le j\le p\}$ are solutions to the equation
\begin{equation}
  \label{eigenf}
  0 = |\lambda I -S_n|=| \lambda I - S_{22}| ~
  |\lambda I -K_n(\la)|~,
\end{equation}
with a random sesquilinear form
\begin{equation}
  K_n(\la)= S_{11} + S_{12}(\lambda I-S_{22})^{-1}S_{21}.
  \label{Klambda}
\end{equation}
Note that the factorization \eq{eigenf} holds for any $\la\notin\mathop{\textrm{spec}}(S_{22})$.
This identity will play a central role in our analysis.



\section{Known results on the spectrum of large sample covariance matrices}\label{sec:known}

\subsection{Mar\v{c}enko-Pastur distributions}\label{ssec:MP}
In this section $y$ is an arbitrary positive constant  and
$H$ an arbitrary  probability measure  on $\R^+$.
Define on the set
\[ \CC^+:= \{ z\in \CC ~:~ \Im(z) >0~    \}~,
\]
the map
\begin{equation}
  \label{map-g}
  g(s) = g_{y,H}(s)= - \frac1 s  +  y \int\!\frac{t}{1+ts} dH(t)~,\quad s\in\CC^+.
\end{equation}
It is well-known (\citep[Chap. 5]{BSbook}) that $g$ is a
one-to-one map from $\CC^+$ onto itself, and
the inverse map $m=g^{-1}$ corresponds to the
Stieltjies transform of  a probability measure $F_{y,H}$ on
$[0,\infty)$.  Throughout the paper and with a small abuse of
  language,
we refer    $F_{y,H}$ as the
Mar\v{c}enko-Pastur (M.P.) distribution  with indexes $(y,H)$.

This family of distributions arises naturally as follows. Consider a
companion matrix $\barSn=\frac1n Z_n^* T_p Z_n$ of the sample
covariance matrix $S_n$. The spectra of $S_n$ and $\barSn$ are
identical except $|n-p|$ zeros. It is then well-known
(\citep{MP},\citep[Chap. 5]{BSbook}) that under Conditions (a)-(d),
the E.S.D. of $\barSn$ converges to the M.P. distribution $F_{y,H}$.
The terminology is slightly ambiguous since the classical M.P.
distribution refers to the limit of the E.S.D. of $S_n$ when
$T_p=I_p$.

Note that   we shall
always extend a function $h$ defined on $\CC^+$ to the real axis
$\R$ by taking the limits
$\lim_{\varepsilon\to 0_+} h(x+i\varepsilon)$ for real $x$'s whenever
these limits exist.  For  $\al \notin\Ga_H$ and  $\al\ne 0$ define
\begin{equation}
  \label{eq:psi}
  \psi(\al)=\psi_{y,H}(\al) := g(-1/\al) =
  \al +    y \al \int\!\frac{t}{\al-t} dH(t)~.
\end{equation}
Note that even though this formula could be extended to  $\al=0$
when $0\notin \Ga_H$, as we will see below that $\al$ is related to the
$-1/m$ where $m$ is a Stieltjies transform, so that there is no
much meaning for  $\al=0$. Therefore,
the point 0 will  always be excluded from the domain of definition of
$\psi$.

Analytical properties of  $F_{y,H} $ can be derived from
the fundamental equation \eq{eq:psi}.
The following lemma, due to    \citet{SilversteinChoi95},
characterizes  the close relationship between
the supports of  the generating measure $H$
and the generated M.P. distribution $F_{y,H}$.

\begin{lemma}\label{lem:supp2}
  If  $ \la  \notin \Ga_{F_{y,H}}$,
  then   $m(\la)\ne 0$ and $\al=-1/m(\la)$  satisfies
  \begin{enumerate}
  \item
    $ \al  \notin \Ga_H $ and  $\al\ne 0$  (so that $\psi(\al)$ is well-defined); \label{alphaitem1}
  \item
    $ \psi'(\al)>0 $.  \label{alphaitem2}
  \end{enumerate}
  Conversely, if $\al$  satisfies {\em \eq{alphaitem1}-\eq{alphaitem2}}, then
  $ \la = \psi(\al) \notin \Ga_{F_{y,H}}$.
\end{lemma}

It is then possible to determine the support of $F_{y,H}$ by looking
at intervals where $\psi'>0$. As an example,
Figure~\ref{fig:psi} displays the function  $\psi$
for the M.P. distribution  with  indexes
$y=0.3$ and  $H$ the uniform distribution on the set
$\{1,4,10\}$.
The function $\psi$ is  strictly
increasing on the following
intervals:
     ($-\infty$, 0), (0, 0.63), (1.40, 2.57)  and   (13.19, $\infty$).
According to Lemma~\ref{lem:supp2}, we get
\[   \Ga_{F_{y,H}} \cap \R^*   =  (0, ~0.32) \cup   (1.37, ~   1.67) \cup  (18.00,~\infty).
\]
Hence, taking into account that 0 belongs to the support of $F_{y,H}$,
we have
\[ \Ga_{F_{y,H}} = \{0\} \cup  [0.32,~ 1.37] \cup
   [1.67,~18.00].
\]

We refer to \citet{BS99b}  for a complete account  of  analytical
properties of the family of   M.P. distributions $\{F_{y,H}\}$
and the maps $\{\psi_{y,H}\}$.  In particular,  the following
conclusions will be useful:
\begin{itemize}
\item
  when restricted to
  $\Ga^c_{F_{y,H}}$, $\psi_{y,H} $ has a well-defined inverse function
  $\psi^{-1}_{y,H}$: $\Ga^c_{F_{y,H}} \to \Ga^c_H$ which is strictly
  increasing;
\item
 the family  $\{F_{y,H}\}$
 is continuous in its  index parameters $(y,H)$ in a wide sense.
 For example,  $\{\psi_{y,H} \}$ tends to the identity function
 as  $y\to 0$.
%
\end{itemize}

\subsection{Exact separation of sample eigenvalues}

We need first quote two  results  of  \citet{BS98,BS99b} on exact separation of
sample eigenvalues.
Recall the ESD's  $(H_n)$ of $(T_p)$,  $y_n=p/n$,  and let
$\{ F_{y_n,H_n} \}$  be the sequence of associated M.P. distributions.
One should not confuse the  M.P. distribution $\{ F_{y_n,H_n} \}$
with the E.S.D. of $\barSn$ although both converge to the
M.P. distribution  $F_{y,H}$ as $n\rw\infty$.

\begin{prop}\label{prop1}
Assume hold Conditions (a)-(d) and the following
\begin{enumerate}
\item [(f)]
  The interval $[a,b]$ with $a>0$
  lies in an open interval $(c,d)$ outside the support of
  $F_{y_n,H_n}$ for all large $n$.
\end{enumerate}
Then
\[
P(\mbox{ no eigenvalue of $S_n$ appears in $[a,b]$ for all large
  $n$ })=1.
\]
\end{prop}

Roughly speaking,  Proposition~\ref{prop1} states that a gap in the
spectra of the $F_{y_n,H_n}$'s  is also a  gap in  the spectrum  of $S_n$ for
large $n$. Moreover, under    Condition (f), we know
by Lemma~\ref{lem:supp2},
that  for large $n$,
\[    \psi_{y_n,H_n}^{-1} \{ [a,b]\}
  \subset \psi_{y_n,H_n}^{-1} \{ (c,d)\}
  \subset \Ga^c_{H_n} ~.
\]
By continuity of $F_{y_n,H_n}$ in its indexes, it follows that we
have for large $n$
\[\psi^{-1} \{ [a,b]\} =  \psi_{y,H}^{-1} \{ [a,b]\}   \subset \Ga^c_{H_n}~.
\]
In other words, it holds almost surely and for large $n$ that,
 $\psi^{-1} \{[a,b]\}$ contains no eigenvalue of $T_p$.
Let for these  $n$, the integer $i_n\geq 0$ be  such that
\begin{equation}
  \label{interval2}
  \textrm{$T_p$ has exactly  $i_n$ eigenvalues larger than }  \psi^{-1}(b)~.
\end{equation}

\begin{prop}\label{prop2}
  Assume Conditions (a)-(d) and (f)  hold.
  If $y[1-H(0)]\leq1$, or $y[1-H(0)]>1$ but $[a,b]$ is
  not contained in $[0,x_0]$ where $x_0>0$ is the smallest value of
  the support of $F_{y,H}$,
  then    with  $i_n$ defined  in \eq{interval2} we have
  \[
  P(  \lambda_{i_n+1}^{S_n} \le a <b \le \lambda_{i_n}^{S_n} \quad\mbox{for all large $n$})=1.
  \]
\end{prop}
In other words, under these conditions,
it happens eventually  that
the numbers of sample eigenvalues  $\{ \lambda_{i}^{S_n}\}$ in both sides
of $[a,b]$ match exactly the numbers of populations eigenvalues $\{\al_k,\be_{n,j}\}$ in both sides
of the interval $\psi^{-1}\{[a,b]\}$.

\section{Almost sure convergence of sample eigenvalues from    generalized spikes}
\label{sec:asconv}

From \eq{eq:psi}, we have 
\[  
 \psi'(\al)= 1-y \int\!\frac{t^2}{(\al-t)^2} dH(t)~,\qquad
 \psi'''(\al)= -6y \int\!\frac{t^2}{(\al-t)^4} dH(t)~.
\]
Therefore, when $\al$ approaches the boundary of the support of $H$,
$\psi'(\al)$ tends to $-\infty$, see also Figure~\ref{fig:psi}.
Moreover,  $\psi'$ is concave on any interval outside $\Ga_H$. 

As we will see, the asymptotic behavior of the sample eigenvalues
generated by a generalized spike eigenvalue $\al$ depends on the sign
of $\psi'(\al)$.   

\begin{defi}
  We call a generalized spike eigenvalue $\al$,
  a {\em distant  spike} for the M.P. law $F_{y,H}$  if   $\psi'(\al)> 0$, and
  a {\em close spike}
  if $\psi'(\al)\le 0$.
\end{defi}

Recall that $\psi$ depend on the parameters $(y,H)$.
When $H$ is  fixed, and since $\psi$ tends to the identity function
as $y\to 0$,  a close spike for a given  M.P. law $F_{y,H}$ becomes a
distant  spike for M.P. law  $F_{y,H}$ for small enough $y$.

As an example, different types of spikes are displayed in
Figure~\ref{fig:psi3}. The solid curve corresponds to a zoomed view
of $\psi_{0.3,H}$ of Figure~\ref{fig:psi}. For $F_{0.3,H}$, the
three values $\al_1$, $\al_2$ and $\al_5$ are close spikes; each
small enough $\al$ (close to zero), or large enough $\al$ (not
displayed), or a value between $u$ and $v$ (see the figure) is a
distant spike. Furthermore, as $y$ decreases from $0.3$ to $0.02$
(dashed curve), $\al_1$, $\al_2$ and $\al_5$ become all distant
spikes.

Throughout this section, for each spike eigenvalue $\al_k$, we denote
by  $\nu_k+1,\ldots,\nu_k+n_k$ the descending ranks of $\al_k$ among
the eigenvalues of $T_p$ (multiplicities of eigenvalues are counted):
in other words, there are $\nu_{k}$ eigenvalues of $T_p$
larger than $\al_k$ and $p-\nu_{k}-n_k$ less.

\begin{thm}\label{thm:main}
  Assume that the conditions (a)-(e) hold. Let  $\al_k$ be a
  generalized spike eigenvalue  of multiplicity $n_k$
  satisfying  $ \psi'(\al_k)>0$ (distant spike) with descending ranks
  $\nu_k+1,\cdots, \nu_k+n_k$.  Then,
  the $n_k$ consecutive sample eigenvalues
  $\{\la^{S_n}_{i}\}$, $i=\nu_k+1,\ldots,\nu_k+n_k$
  converge almost surely to  $\psi(\al_k)$.
\end{thm}

\begin{proof} Recall Figure~\ref{fig:psi3} of the $\psi$ function, for each
distant spike $\al_k$, there is an interval $(u_k,v_k)$ such that
\begin{itemize}
\item $u_k<\al_k<v_k$;
\item $\psi'(u_k)=\psi'(v_k)=0$;
\item $\psi'(\al)>0$ for all $\al\in(u_k,v_k)$.
\end{itemize}
Here we make the convention that $v_k=\infty$ if $\psi'(\al)>0$ for
all $\al>\al_k$ and $u_k=0$ if $\psi'(\al)>0$ for all $\al\in
(0,\al_k)$.

Recall that the support of $F_{y_n,H_n}$ is determined by
\begin{equation}
  \label{eq:psin}
  \psi_n'(\al)=\psi_{y_n,H_n}'(\al)  =
  1 -    y_n \bigg[\frac{p'}{p}\int\!\frac{t^2}{(\al-t)^2}
  dH_n^v(t)~  + \frac1p \sum_{j=1}^K\frac{n_j\al_j^2}{(\al-\al_j)^2}\bigg],
\end{equation}
where $H_n^v = \frac1{p'}\sum_{j} \de_{\be_{n,j}} $ is the ESD of $V_p$.

Let $\tilde v_k=\min(v_k,\al_{k-1})$ if $k>1$ and $\tilde v_k=v_k$
otherwise. Choose $v,v'$ and $\al'_u,\al_u$ such that
$\al_k<\al'_u<\al_u<v<v'<\tilde v_k$. By condition (e), all
eigenvalues of $T_p$ will keep away from the interval $(\al_u',v')$
for all large $n$. Thus, $\psi_n'(\al)\to\psi'(\al)>0$ uniformly on
the interval $[\al_u',v']$. Hence, the interval
$(\psi(\al_u'),\psi(v'))$ will be out of the support of
$F_{y_n,H_n}$ for all large $n$. Consequently, the interval
$[\psi(\al_u),\psi(v)]$ satisfies the conditions of Proposition
\ref{prop2} with $i_n=\nu_k$. Therefore, by Proposition \ref{prop2},
we have
$$
\begin{cases}
  P(\lambda_{\nu_k+1}^{S_n}\le \psi(\al_u)<\psi(v)\le
  \lambda_{\nu_k}^{S_n}, ~\mbox{ for all large $n$})=1&
  \mbox{ if }\nu_k>0;\cr
  P(\lambda_{\nu_k+1}^{S_n}\le \psi(\al_u),\
  \mbox{ for all large $n$})=1 & \mbox{ otherwise. }\cr
\end{cases}
$$
Therefore, it holds almost surely
\[ \limsup_n \lambda_{\nu_k+1}^{S_n}\le \psi(\al_u)  ,
\]
and finally, letting $\al_u \rw \al_k$,
\begin{equation}\label{upperB}
  \limsup_n \lambda_{\nu_k+1}^{S_n}\le \psi(\al_k).
\end{equation}

Similarly, one can prove that for any $\tilde u_k<u<\al_l<\al_k$,
$$
\begin{cases}
  P(\lambda_{\nu_k+n_k+1}^{S_n}\le \psi(u)  <\psi(\al_l)\le\lambda_{\nu_k+n_k}^{S_n},
  \ \mbox{ for all large  $n$})=1 &
  \mbox{ if } \nu_k+n_k<p,\cr
  P(\lambda_{\nu_k+n_k}^{S_n}\ge \psi(\al_l),\ \mbox{ for all large  $n$})=1&
  \mbox{ otherwise,}
\end{cases}
$$
where $\tilde u_k=\max(u_k,\al_{k+1})$ if $k<K$ and $\tilde u_k=u_k$
otherwise.

Consequently,
\begin{equation}\label{lowerB}
  \liminf_n \lambda_{\nu_k+n_k}^{S_n}\ge  \psi(\al_k).
\end{equation}
Thus, we proved that almost surely,
$$
\lim_n \lambda_{\nu_k+j}^{S_n}=\psi(\al_k), \mbox{ for } j=1,\cdots,n_k.
$$
The proof of Theorem \ref{thm:main} is complete.
\end{proof}

Next we consider close spikes.


\begin{thm}\label{thm:close-spikes}
  Assume that the conditions (a)-(e) hold. Let  $\al_k$ be a
  generalized spike eigenvalue of multiplicity $n_k$
  satisfying  $ \psi'(\al_k)\le 0$ (close spike) with descending ranks
  $\nu_k+1,\ldots,\nu_k+n_k$. Let $I$ be the
  maximal interval in $\Ga^c_H$ containing $\al_k$.

  \begin{enumerate}
  \item
    If $I$ has a sub-interval $(u_k,v_k)$ on which $\psi'>0$ (then we
    take this interval to be maximal), then
    the $n_k$ sample eigenvalues $\{\la^{S_n}_{j}\}$, $j=\nu_k+1,\ldots,\nu_k+n_k$
    converge almost surely to the number $\psi(w)$
    where $w$ is  one of the  endpoints  $\{u_k,v_k\}$ nearest to $\al_k$ ;
  \item
    If for all $\al\in I$, $\psi'(\al)\le 0$, then
    the $n_k$ sample eigenvalues $\{\la^{S_n}_{j}\}$, $j=\nu_k+1,\ldots,\nu_k+n_k$
    converge almost surely to the $\ga$-th quantile of $G$, the L.S.D. of
    $S_n$,   where $\ga=H(0, \al_k)$.
  \end{enumerate}
\end{thm}

\begin{proof}
The proof refers to the curves of Figure~\ref{fig:psi3}.

\medskip
\noindent(i).\quad 
Suppose $\alpha_k$ is a spike eigenvalue
satisfying $\psi'(\alpha_k)\le 0$ and there is an interval
$(u_k,v_k)\subset I$ on which $ \psi'>0$ ($\al_k$ is like the $\al_1$ on
the figure). According to Lemma~\ref{lem:supp2}, 
$\psi\{(u_k,v_k)\}\subset\Ga^c_{F_{y,H}}$ and $\psi(u_k)$ is a boundary point of the
support of $G$, the L.S.D. of $S_n$. Without loss of generality, we
can assume $\al_k\le u_k$, the argument of the other situation where
$\al_k>v_k$ being similar.

Choose $u_k<\al_u<v<\tilde v$ ($\tilde v=\min(v_k,\al_{k-1})$ or $v_k$ in
accordance with $k>1$ or not) such that $(\al_u,v)\subset I$,
by the argument used in the proof of Theorem \ref{thm:main}, one can
prove that
$$
\begin{cases}
  P(\lambda_{\nu_k+1}^{S_n}\le \psi(\al_u)<\psi(v)\le
  \lambda_{\nu_k}^{S_n}, ~\mbox{ for all large $n$})=1&
  \mbox{ if }\nu_k>0;\cr
  P(\lambda_{\nu_k+1}^{S_n}\le \psi(\al_u),\
  \mbox{ for all large $n$})=1 & \mbox{ otherwise. }\cr
\end{cases}
$$
This proves that almost surely, 
$$
\limsup \lambda_{\nu_k+1}^{S_n}\le \psi(u_k) \le\liminf \lambda_{\nu_k}^{S_n}~. 
$$
On the other hand, 
since  $\psi(u_k)$ is a boundary point of the
support of $G$, we know that for any $\veps>0$, almost surely, 
the number of $\la_i^{S_n}$'s falling into  $[\psi(u_k)-\veps,\psi(u_k)]$
tends to  infinity. Therefore, 
$$
\liminf \lambda_{\nu_k+n_k+1}^{S_n}\ge \psi(u_k)-\veps,\quad \textrm{a.s.}.
$$
Since $\veps$ is arbitrary, we have finally proved that almost surely,  
$$
\lim \lambda_{\nu_k+j}^{S_n}= \psi(u_k),\quad  j=1,\cdots,n_k.
$$
Thus, the proof of Conclusion (i) of Theorem \ref{thm:close-spikes} is
complete.

Similarly, if the spiked eigenvalue $\alpha_k$ is like $\alpha_2$, we can show that the $n_k$
corresponding eigenvalues of $S_n$ goes to
$\psi(v_k)$.

\medskip
\noindent
(ii) \quad
If the spiked eigenvalues is like $\alpha_5$, where the gap of support of LSD
disappeared, clearly
the corresponding sample eigenvalues $\la_{\nu_k+1},\ldots,\la_{\nu_k+n_k}$  tend to the $\gamma$-th quantile
of the LSD of $S_n$ where
\[  \gamma = 1-\lim \frac{i_n}{\nu_k} = H(0,\al_k).
\]
\end{proof}

\subsection{Case of Johnstone's spiked population model}
\label{ssec:Johnstone}
In the case of Johnstone's model, $H$ reduces to the Dirac mass
$\de_{1}$ and the LSD $G$ equals the Mar\v{c}enko-Pastur law with
$\Ga_G=[a_y,b_y]$. Each $\al>0$, $\al\ne 1$ is then a spike
eigenvalue.
 The associated  function $\psi$  in \eq{eq:psi}  becomes
\begin{equation}\label{eq:psi-null}
  \psi(\al_k) = \al_k + \frac{y\al_k}{\al_k-1}.
\end{equation}
The function $\psi$ has the following properties, see
Figure~\ref{fig:psiJohnstone}:
\begin{itemize}
  \item its range equals $(-\infty,a_y]\cup[b_y,\infty)$ ;
  \item $\psi(1-\sqrt y)=a_y$ ,~ $\psi(1+\sqrt y)=b_y$;
  \item $\psi'(\al) > 0 \Leftrightarrow  |\al-1|> \sqrt y $.
\end{itemize}
Therefore, by Theorem~\ref{thm:main},  for any  spike eigenvalue
satisfying $\al_k>1+\sqrt y$ (large enough) or   $\al_k<1-\sqrt y$
(small enough), there is a
packet of $n_k$ consecutive eigenvalues $\{\la_{n,j}\}$ converging
almost surely to     $ \psi(\al_k) \notin [a_y,b_y] $.
In other words,  assume there are exactly $K_1$
spikes  greater than $1+\sqrt y$ and $K_2$ spikes smaller than
$1-\sqrt y$.
By Theorems~\ref{thm:main} and \ref{thm:close-spikes}
we conclude that 
\begin{enumerate}
\item  the $N_1:=n_1+\ldots+n_{K_1}$ largest eigenvalues
  $\{\la_j^{S_n}\}$, $j=1,\ldots,N_1$   tend to
  their respective limits $\{\psi(\al_k)\}$, $k=1,\ldots,K_1$ ;
\item
  the immediately following  largest eigenvalue $\la_{N_1+1}^{S_n}$
  tends to the right edge $b_y$;
\item  the $N_2:=n_K + \cdots+n_{K-K_2+1}$ smallest sample eigenvalues
  $\{\la_{n,p-j}^{S_n}\}$, $j=0,\ldots,N_2-1$   tend to
  their respective limits $\{\psi(\al_k)\}$, $k=K,\ldots,K-K_2+1$ ;
\item
  the immediately following smallest eigenvalue
  $\la_{p- N_2}^{S_n}$ tends to the left edge $a_y$.
\end{enumerate}

Hence we have recovered  the content of Theorem~1.1 of \cite{BaikSilv06}.

\subsection{An example of generalized spike eigenvalues}

Assume that $T_p$ is diagonal with three  base eigenvalues
 $\{1,4,10\}$,   nearly $p/3$ times for each of them, and
there are four  spike eigenvalues $(\al_1,\al_2,\al_3,\al_4)=(15,~6,~2,~0.5)$, 
 with respective multiplicities $(n_k)=(3,2,2,2)$.
The limiting population-sample ratio is taken to be  $y=0.3$.
The limiting population spectrum  $H$ is then the uniform distribution on $\{1,4,10\}$.
The support of the limiting   Mar\v{c}enko-Pastur distribution
$F_{0.3,H}$   contains two intervals [0.32, 1.37] and [1.67, 18], see  \S\ref{ssec:MP}.
The  $\psi$-function of \eq{eq:psi} for the current case is
displayed in  Figure~\ref{fig:psi}. 
For simulation, we use $p'=600$ so that $T_p$ has the following 609
eigenvalues:
\[ 15,~15,~15, 
  ~\underbrace{10,\ldots,10}_{200},
  ~6,~6,
  ~\underbrace{4,\ldots,4}_{200},
  ~2,~2,
  \underbrace{1,\ldots,1}_{200},
  ~0.5,~0.5~.
\]
From the table 
\begin{center}
  \begin{tabular}{r|llll}
    spike        $\al_k$   &  15    &  6     &  2   & 0.5  \\ \hline    
  multiplicity   $n_k$     &   3    &  2     &  2   &  2  \\
   $\psi'(\al_k)$          &  $ +$    &  $-$     &  $+$   &  $-$  \\
   $\psi(\al_k)$           &  18.65    &  5.82    &  1.55  &  0.29 \\ 
  descending ranks         &  1, 2, 3    & 204, 205 & 406, 407 & 608, 609   \\ \hline
\end{tabular}
\end{center}
\medskip
we see that 6 is a close spike for  $H$  while  the three others are distant
ones. By Theorems~\ref{thm:main} and \ref{thm:close-spikes}, we know
that 
\begin{itemize}
\item the 7 sample eigenvalues $\la_j^{S_n}$ with $j\in\{1,~2,~3,~406,
  ~407,~608, ~609\}$ associated to distant spikes tend to  18.65,
  1.55  and   0.29, respectively, which are located outside the
  support of limiting distribution $F_{0.3,H}$ (or $G$); 
\item
  the two sample eigenvalues $\la_j^{S_n}$ with  $j=204,205$ 
  associated to the close spike $6$  tend to a limit located inside
  the support, the $\ga$-th quantile of
  the limiting distribution $G$  where $\ga=H(0,6)=2/3$.
\end{itemize}
There facts are illustrated  by a simulation 
 sample displayed in Figure~\ref{fig:example}.

\section{CLT for  sample eigenvalues from distant generalized spikes}
\label{sec:CLT}

Following  Theorem~\ref{thm:main}, to any distant generalized spike eigenvalue
$\al_k$, there is  a packet of
$n_k$ consecutive sample eigenvalues
 $\{\la_j^{S_n} :~  j \in J_k   \} $
converging  to $\psi(\al_k)\notin\Ga_G$ where $J_k$ are the descending
ranks of $\al_k$ among the eigenvalues of $T_p$ (counting
multiplicities).
The aim of this section  is to derive a CLT for  $n_k$-dimensional
vector
\[     \sqrt n \{ \la_j^{S_n} -\psi(\al_k)  \} ~,\quad  j\in J_k.
\]
The method follows \citet{BaiYaoIHP} which considers  Johnstone's spiked
population model.
Consider the random form $K_n$
introduced in \eq{Klambda} and  let
\begin{equation}
  A_n=(a_{ij}) =A_n(\la)= X_2^*(\lambda I-X_2X_2^*)^{-1}X_2, \quad
  \la\notin \Ga_G .
\label{eq:An}
\end{equation}
By Lemma~\ref{lem2},  detailed in \S \ref{sec:proofs},
we know that
$  n^{-1} trA_n$, $n^{-1}trA_nA_n^* $ and
$ n^{-1}\sum_{i=1}^na_{ii}^2 $ converge,  almost surely or in
probability,  to
$ y m_1(\la)$,   $ y m_2(\la)$ and
$    \left( { y[1+m_1(\la)] } / \{\lambda - y [1+m_1(\la)] \}
\right)^2$, respectively. Here, the $m_j(\la)$ are
some specific transforms of the LSD $G$ (see   \S\ref{sec:proofs}).

Therefore, the random form $K_n$ in
\eq{Klambda} can be decomposed as follows
\begin{eqnarray*}
  K_n(\lambda) &=& S_{11} + X_1A_n X_1^*  =  \frac1n  \xi_{1:n}(I+A_n)\xi_{1:n}^*\\
  &=&  \frac1n  \left\{ \xi_{1:n}(I+A_n)\xi_{1:n}^* -  \Sigma \mathrm{tr}(I+A_n)
  \right\}  + \frac1n  \Sigma   \mathrm{tr}(I+A_n)   \\
  & = &  \frac1{\sqrt n}  R_n + \left[ 1+y m_1(\lambda)\right] \Sigma
  +o_P( \frac1{\sqrt n} ),
\end{eqnarray*}
with
\begin{equation}
 R_n = R_n(\la) = \frac1{\sqrt n}  \left\{ \xi_{1:n}(I+A_n)\xi_{1:n}^* -
  \Sigma \mathrm{tr}(I+A_n)
  \right\}.       \label{Rn}
\end{equation}
In the last derivation, we have used the fact
\[ \frac1n tr (I+A_n) = 1 + ym_1(\la)+ o_P( \frac1{\sqrt n} ),
\]
which follows from a  CLT for $\mathrm{tr}(A_n)$
\citep[see][]{BS04}.

For the statement of our result, we first need
 to find the limit distribution of the sequence of
random matrices $\{R_n(\la)\}$.
The situation is different for  the
real and complex cases.  By applications of Propositions 3.1 and 3.2
in \cite{BaiYaoIHP}, we have for $\la\notin\Ga_G$, 
\begin{enumerate}
\item if the variables $(w_{ij})$ are real-valued,  
  the random matrix $R_n(\la)$ converges weakly to a
  symmetric random matrix $R(\la)=(R_{ij}(\la))$   with  zero-mean Gaussian
  entries  having  an explicitly known  covariance function ; 
\item 
  if the variables $(w_{ij})$ are complex-valued,  
  the random matrix  $R_n$ converges weakly to a
  zero-mean Hermitian random matrix  $R(\la)=(R_{ij}(\la))$. Moreover,
  the real and imaginary  parts of its upper-triangular bloc
  $\{ R_{ij}(\la),~1\le i\le j\le M\} $ form a $2K$-dimensional Gaussian
  vector with an explicitly  known covariance matrix.
\end{enumerate}

We are in order to introduce our CLT.
Let the spectral decomposition of $\Si$,
\begin{equation}\label{decompSigma}
    \Si = U
    \begin{pmatrix}
      \al_1 I_{n_1} & \cdots  &   0\\
      0             & \ddots  &   0 \\
      \cdots        &  0      & \al_K I_{n_K}
    \end{pmatrix}
    U^* ~,
\end{equation}
where $U$ is an unitary matrix.
Let  $\psi_k=\psi(\al_k)$ and   $R(\psi_k)$ be the
weak Gaussian  limit of the sequence of matrices of random forms
$[R_n(\psi_k)]_n$  recalled above (in both real and complex variables
case). Let
\begin{equation} \label{bloc-R-limit}
  \widetilde R (\psi_k) =  U^* R(\psi_k)U~.
\end{equation}

\begin{thm}\label{thm:CLT}
  For each distant generalize  spike  eigenvalue,
  the $n_k$-dimensional real vector
  \[  \sqrt n \{  \la_j^{S_n}-\psi_k ,~    j\in J_k\} ~,
  \]
  converges weakly to the distribution of the $n_k$  eigenvalues of the
  Gaussian random matrix
  \[     \frac1{1+ym_3(\psi_k)\al_k} \widetilde R_{kk}  (\psi_k).
  \]
  where
  $\widetilde R_{kk}  (\psi_k)$  is the  $k$-th diagonal block of
  $\widetilde R (\psi_k)$
  corresponding to the indexes $ \{  u,v  \in J_k\}$.
\end{thm}

It is worth noticing   that the limiting distribution
of such $n_k$ packed  sample extreme  eigenvalues  are generally
{\em  non Gaussian} and asymptotically dependent.
Indeed,  the limiting
distribution of  a single sample extreme eigenvalue
$\la_j^{S_n}$ is  Gaussian if and only if the corresponding generalized
spike eigenvalue is simple.  We refer the reader to \cite{BaiYaoIHP}
for detailed examples  illustrating these same facts but for Johnstone's
model.

\section{Lemmas}
\label{sec:proofs}

For $\la \notin \Ga_G$, we
define
\begin{eqnarray*}
  m_1(\lambda) & =&   \int \frac{x }{\lambda-x}  dG(x) ,  \label{m1} \\
  m_2(\lambda) & =&   \int  \frac{x^2 }{(\lambda-x)^2}   dG(x)  \label{m2} ~,\\
  m_3(\lambda) & =&   \int  \frac{x }  {(\lambda-x)^2}   dG(x)  \label{m3} ~.
\end{eqnarray*}
The following lemma gives the law of large numbers for some useful
statistics of   $A_n$ defined in \eq{eq:An}.
 We omit its proof because it is a
straightforward extension of Lemma 6.1 of \cite{BaiYaoIHP},
related to Johnstone's spiked population model, to the present
generalized spiked population model.

\begin{lemma}\label{lem2}
  Under the assumptions of Theorem \ref{thm:main}, for all $\lambda\in[a,b]$, we have
  \begin{eqnarray}
    \frac{1}{n}trA_n & \stackrel{a.s.}\longrightarrow &  y m_1(\la)~,\label{a14bis}\\
    \frac{1}{n}trA_nA_n^* & \stackrel{a.s.}\longrightarrow &  y m_2(\la)~,\label{a14}\\
    \frac{1}{n}\sum\limits_{i=1}^na_{ii}^2  & \stackrel{a.s.}\longrightarrow&
    \left( \frac{ y[1+m_1(\la)] }  {\lambda - y [1+m_1(\la)] } \right)^2.
    \label{a15}
  \end{eqnarray}
\end{lemma}

\begin{lemma}\label{lmm:K}
  For all $\la\in [a,b]$, $K_n(\la)$ converges
  almost surely  to
  the constant matrix $[1+ym_1(\la)]\Si$.
\end{lemma}
\begin{proof}
The random form $K_n$ in
\eq{Klambda} can be decomposed as follows
\[
  K_n(\lambda) = S_{11} + X_1A_n X_1^*  =
  \frac1n  (\xi_1,\ldots,\xi_n ) (I+A_n)(\xi_1,\ldots,\xi_n )^* .
\]
Define $M$ be the event that $S_{22}$ has no eigenvalues in the
interval $[a',b']$ which satisfies $[a,b]\subset (a',b')$ and
$[a',b']\subset (c,d)$. On the event $M$, the norm of $A_n$ is
bounded by $\max\{\frac1{a-a'},\frac1{b'-b}\}$. By independence, it
is easy to show that
$$
\frac1n  \left\{(u_1,\ldots,u_n ) (I+A_n)(u_1,\ldots,u_n )^*I_M-
  [\mathrm{tr}(I+A_n)]I_M \right\}\stackrel{a.s.}\to 0.
  $$
By proposition \ref{prop1}, $I_m\to 1,a.s.$. Thus
\begin{eqnarray}
\label{Dn}
  D_n(\la) & = & o_{a.s.}(1)+
  +  [\frac1n \mathrm{tr}(I+A_n)] \Sigma I_M\stackrel{a.s.}\to
  (1+ym_1(\la))\Sigma,
\end{eqnarray}
where the last step follows from \eq{a14bis}. The conclusion
follows.
\end{proof}

\bibliographystyle{plainnat}
\bibliography{matrices}

\vfill

\begin{figure}[ht]
  \begin{center}
    \includegraphics[width=0.5\textwidth,height=0.96\textwidth,angle=-90]
                    {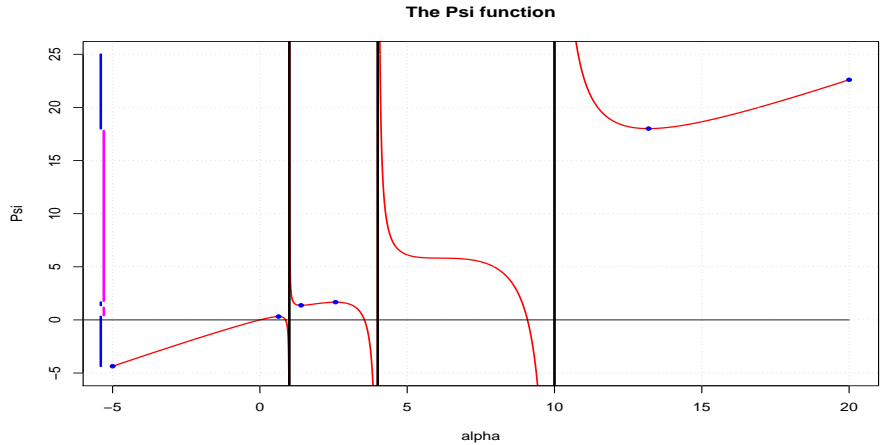}
  \end{center}
  \caption{\label{fig:psi}
     The $\psi$ function for the Mar\v{c}enko-Pastur distribution
     $F_{0.3,H}$ with $H$ the uniform distribution on the set
     $\{1,4,10\}$.
     Blue points indicate intervals where $\psi'>0$. Singular points
     of $\psi$ are indicated as
     vertical lines corresponding to the support  of $H$. On the
     left,
     the support set of $F_{0.3,H}$ (except the point 0) and its complementary set are
     indicated as magenta and blue segments respectively.}
\end{figure}


\begin{figure}[ht]
  \begin{center}
    \includegraphics[width=0.5\textwidth,height=0.96\textwidth,angle=-90]
                    {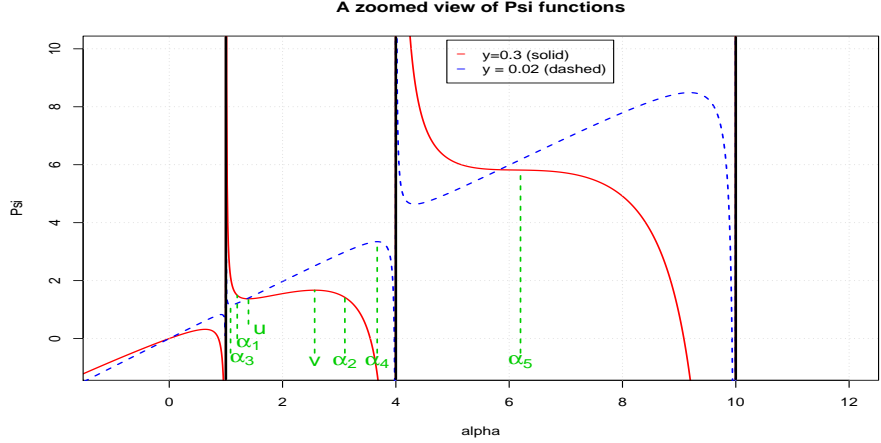}
  \end{center}
  \caption{\label{fig:psi3}
     A zoomed view  of the $\psi$ functions for the Mar\v{c}enko-Pastur distribution
     $F_{0.3,H}$ (solid curve) and $F_{0.02,H}$ (dashed curve) with
     $H$ the uniform distribution on the set
     $\{1,4,10\}$.  The  three   points
     $\al_{1}$, $\al_{2}$  and $\al_{5}$  are close spikes for $F_{0.3,H}$ where
     $\psi_{0.3,H}'\le 0$. They become all
     distant spikes for $F_{0.02,H}$ as $\psi_{0.02,H}' > 0$.}
\end{figure}

\begin{figure}[htp]
  \begin{center}
    \includegraphics[width=0.5\textwidth,height=0.9\textwidth,angle=-90]
                    {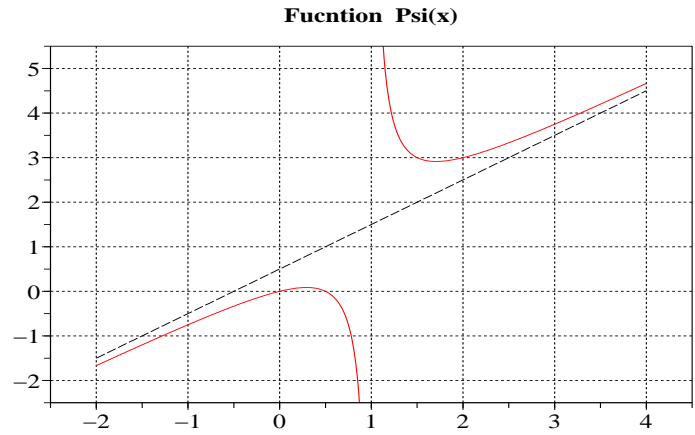}
  \end{center}
                  \caption{\label{fig:psiJohnstone}
    The function $\al\mapsto \psi(\al)=\al+ y\al/(\al-1)$ which maps a
    spike  eigenvalue $\al$ to the limit of an associated sample
    eigenvalue in Johnstone's spiked population model.\quad Figure with
    $y=\frac12$;  ~$[1\mp\sqrt y]=[0.293,~1.707]$;~$[(1\mp\sqrt y)^2]=[0.086,~2.914]$           .}
\end{figure}

\begin{figure}[htp]
  \begin{center}
    \includegraphics[width=0.3\textwidth,height=0.96\textwidth,angle=-90]
                    {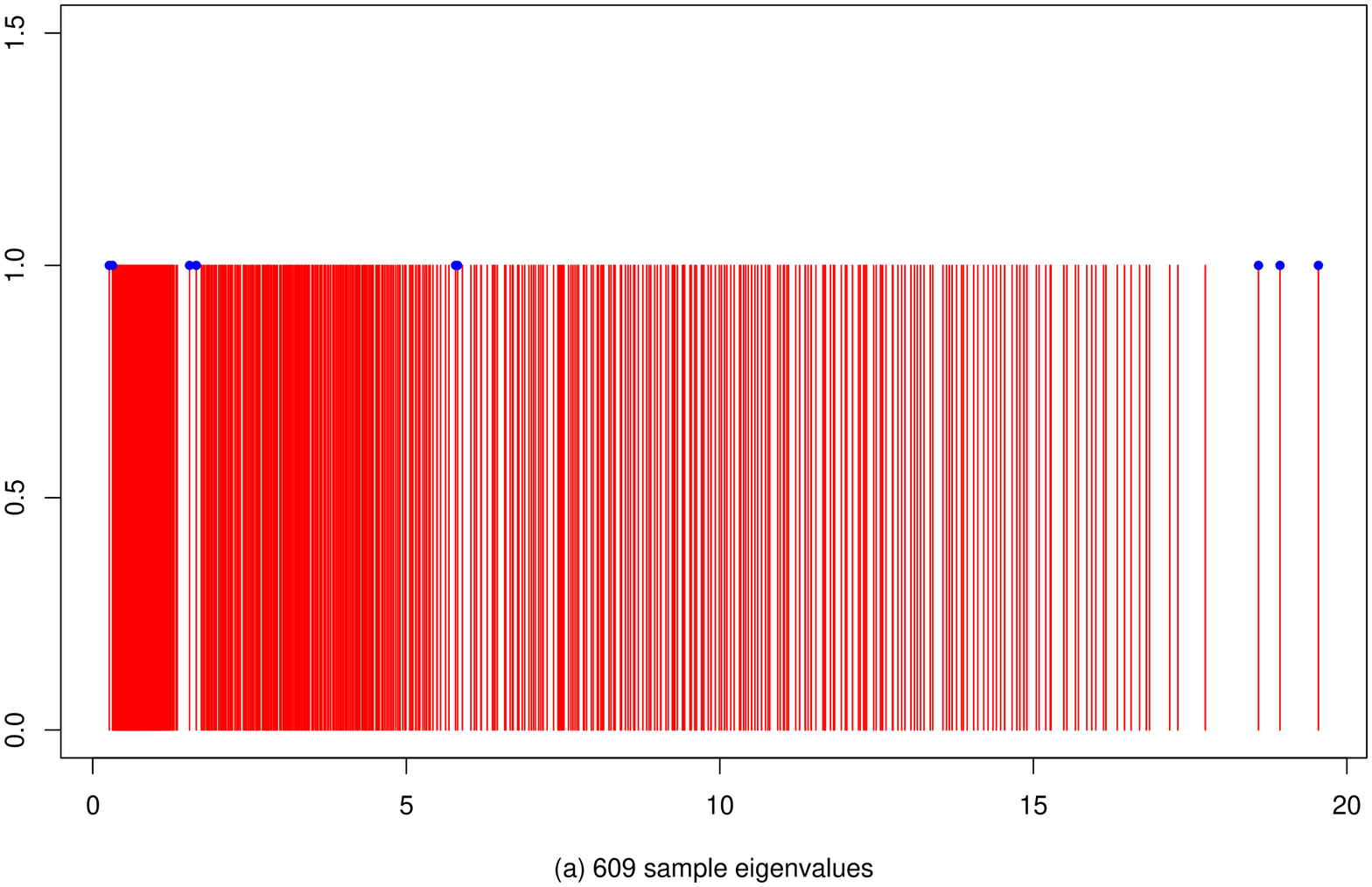}\\
    \includegraphics[width=0.3\textwidth,height=0.96\textwidth,angle=-90]
                    {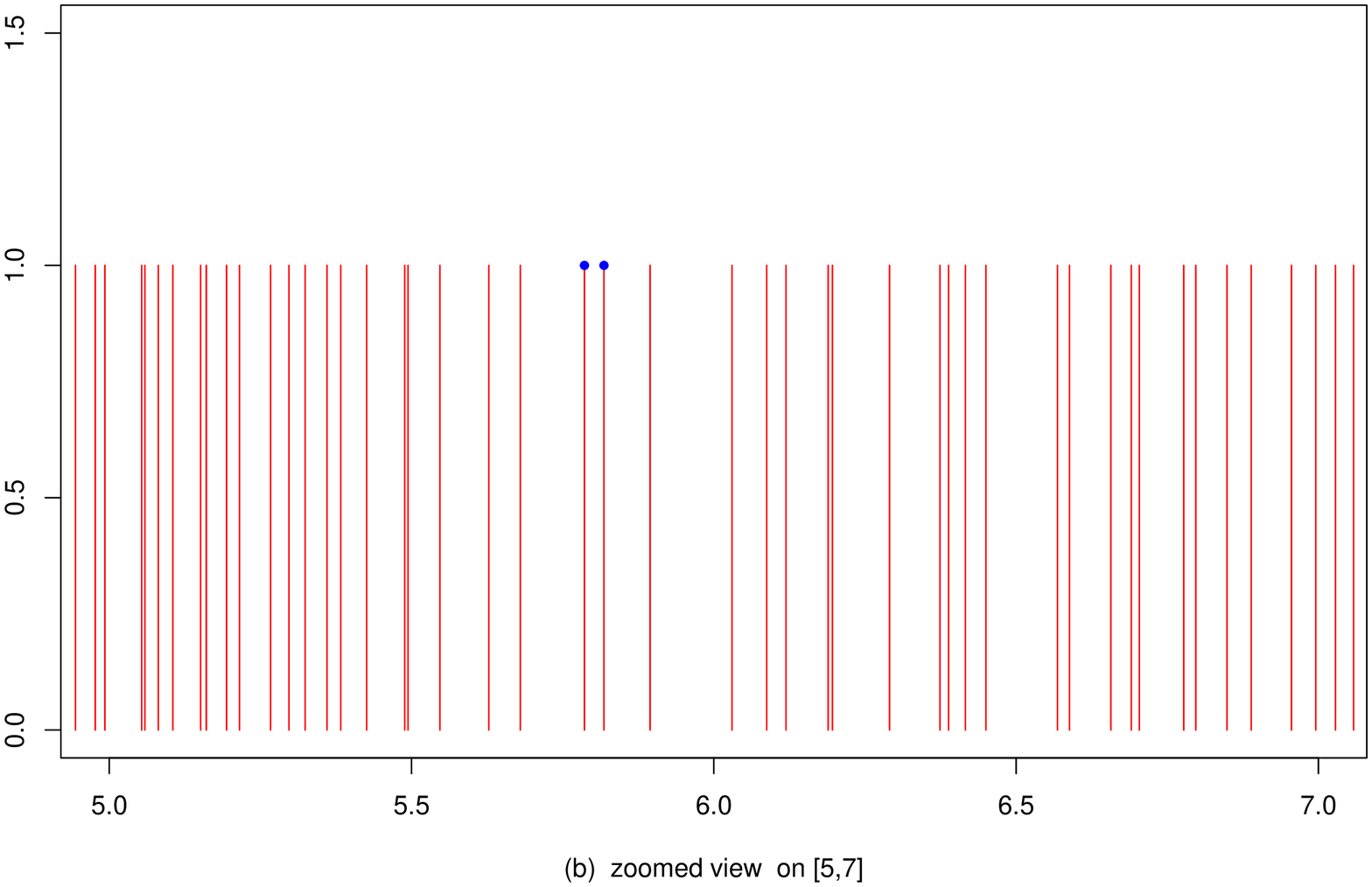}\\
    \includegraphics[width=0.3\textwidth,height=0.96\textwidth,angle=-90]
                    {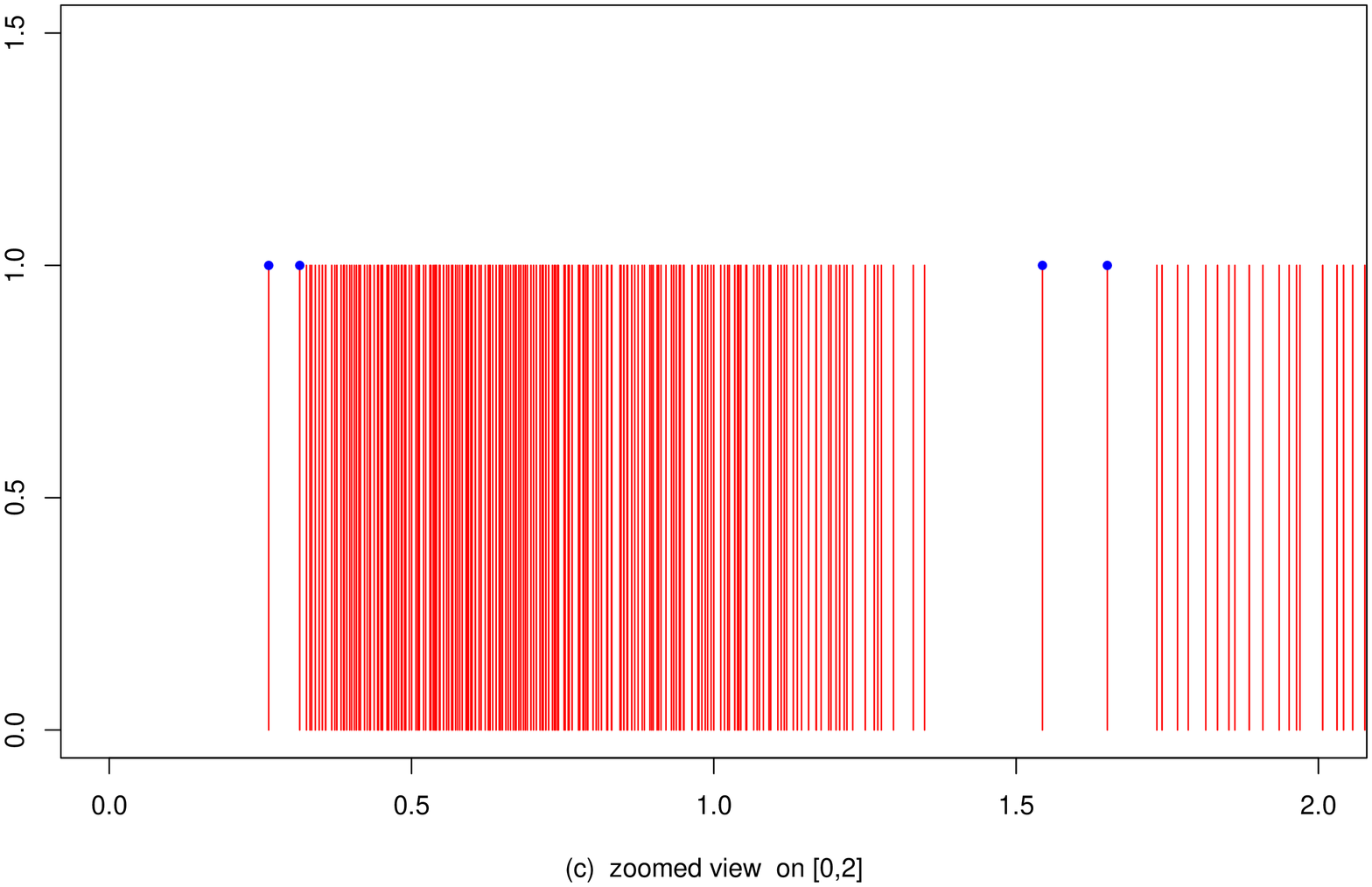}
  \end{center}
  \caption{\label{fig:example}
    An example of $p=609$ sample eigenvalues (a),  and  two zoomed
    views (b) and (c)  on
    [5,7]  and [0,2] respectively. The limiting distribution of the E.S.D has support
    $[0.32,  1.37] \cup  [1.67, 18.00]$.  The 9  sample  eigenvalues
    $\{\la_j^{S_n},~j=1,2,3,204,205,406,407,608,609~\}$ associated to
    the spikes are marked with a blue point. Gaussian entries.}
\end{figure}

\end{document}